\documentclass[10pt]{amsart}
\usepackage{amssymb,MnSymbol,cite}
\usepackage{amsthm,amsmath}

\title[]{On modular decompositions of system signatures}

\author{Jean-Luc Marichal}
\address{Mathematics Research Unit, FSTC, University of Luxembourg \\
6, rue Coudenhove-Kalergi, L-1359 Luxembourg, Luxembourg}
\email{jean-luc.marichal[at]uni.lu }

\author{Pierre Mathonet}
\address{University of Li\`ege, Department of Mathematics, Grande Traverse, 12 - B37, B-4000 Li\`ege, Belgium}
\email{p.mathonet[at]ulg.ac.be }

\author{Fabio Spizzichino}
\address{Department of Mathematics, University La Sapienza, Piazzale A.\ Moro 2, I-00185, Rome, Italy}
\email{fabio.spizzichino[at]uniroma1.it }

\date{October 10, 2014}

\begin{document}

\theoremstyle{plain}
\newtheorem{theorem}{Theorem}
\newtheorem{lemma}[theorem]{Lemma}
\newtheorem{proposition}[theorem]{Proposition}
\newtheorem{corollary}[theorem]{Corollary}
\newtheorem{fact}[theorem]{Fact}
\newtheorem*{main}{Main Theorem}

\theoremstyle{definition}
\newtheorem{definition}[theorem]{Definition}
\newtheorem{example}[theorem]{Example}

\theoremstyle{remark}
\newtheorem*{conjecture}{Conjecture}
\newtheorem{remark}{Remark}
\newtheorem{claim}{Claim}

\newcommand{\N}{\mathbb{N}}                     
\newcommand{\R}{\mathbb{R}}                     
\newcommand{\bfa}{\mathbf{a}}
\newcommand{\bfx}{\mathbf{x}}
\newcommand{\bfy}{\mathbf{y}}
\newcommand{\bfz}{\mathbf{z}}
\newcommand{\bfb}{\mathbf{b}}
\newcommand{\tk}{{\mathcal T}_k}
\newcommand{\uno}{\boldsymbol 1}

\begin{abstract}
Considering a semicoherent system made up of $n$ components having i.i.d.\ continuous lifetimes, Samaniego defined its structural signature as the $n$-tuple whose $k$-th coordinate is the probability that the $k$-th component failure causes the system to fail. This $n$-tuple, which depends only on the structure of the system and not on the distribution of the component lifetimes, is a very useful tool in the theoretical analysis of coherent systems.

It was shown in two independent recent papers how the structural signature of a system partitioned into two disjoint modules can be computed from the signatures of these modules. In this work we consider the general case of a system partitioned into an arbitrary number of disjoint modules organized in an arbitrary way and we provide a general formula for the signature of the system in terms of the signatures of the modules.

The concept of signature was recently extended to the general case of semicoherent systems whose components may have dependent lifetimes. The same definition for the $n$-tuple gives rise to the probability signature, which may depend on both the structure of the system and the probability distribution of the component lifetimes. In this general setting, we show how under a natural condition
on the distribution of the lifetimes, the probability signature of the system can be expressed in terms of the probability signatures of the modules. We finally discuss a few situations where this condition holds in the non-i.i.d.\ and nonexchangeable cases and provide some applications of the main results.
\end{abstract}



\maketitle

\section{Introduction}

We consider an $n$-component system $S=(C,\phi,F)$, where $C$ is the set $[n]=\{1,\ldots,n\}$ of components, $\phi\colon\{0,1\}^n\to\{0,1\}$ is the
structure function (which expresses the state of the system in terms of the states of its components), and $F$ denotes the joint c.d.f.\ of the
component lifetimes $T_1,\ldots,T_n$, that is,
$$
F(t_1,\ldots,t_n) ~=~ \Pr(T_1\leqslant t_1,\ldots,T_n\leqslant t_n)\, ,\qquad t_1,\ldots,t_n\geqslant 0.
$$
We assume that the system is \emph{semicoherent}, i.e., the structure function $\phi$ is nondecreasing\footnote{Such a function is also known as a monotone function.} in each variable and satisfies the
conditions $\phi(0,\ldots,0)=0$ and $\phi(1,\ldots,1)=1$. We also assume that the c.d.f.\ $F$ has no ties, that is, $\Pr(T_i=T_j)=0$ for all
distinct $i,j\in [n]$.

The concept of \emph{signature} was introduced in 1985 by Samaniego \cite{{Sam85}}, for systems whose components have continuous and i.i.d.\
lifetimes, as the $n$-tuple $\mathbf{s}=(s_1,\ldots,s_n)$ whose $k$-th coordinate $s_k$ is the probability that the $k$-th component failure causes the system to
fail. In other words, we have
$$
s_k ~=~ \Pr(T_S=T_{k:n}),\qquad k\in [n],
$$
where $T_S$ denotes the system lifetime and $T_{k:n}$ denotes the $k$-th smallest lifetime, i.e., the $k$-th order statistic obtained by
rearranging the variables $T_1,\ldots,T_n$ in ascending order of magnitude.

It was shown in \cite{Bol01} that $s_k$ can be explicitly written in the form\footnote{As usual, we identify Boolean vectors $\bfx\in\{0,1\}^n$
and subsets $A\subseteq [n]$ by setting $x_i=1$ if and only if $i\in A$. We thus use the same symbol to denote both a function
$f\colon\{0,1\}^n\to\R$ and the corresponding set function $f\colon 2^{[n]}\to\R$, interchangeably. For instance, we write $\phi(0,\ldots,0)=\phi(\varnothing)$.}
\begin{equation}\label{eq:asad678}
s_k ~=~ \sum_{\textstyle{A\subseteq C\atop |A|=n-k+1}}\frac{1}{{n\choose |A|}}\,\phi(A)-\sum_{\textstyle{A\subseteq C\atop
|A|=n-k}}\frac{1}{{n\choose |A|}}\,\phi(A)\, .
\end{equation}
This formula shows that, in the i.i.d.\ case, the probability $\Pr(T_S=T_{k:n})$ does not depend on the distribution $F$ of the component lifetimes. Thus, the system signature is a purely combinatorial object associated with the structure $\phi$. Due to this feature, in both the i.i.d.\ and non-i.i.d.\ cases the $n$-tuple $\mathbf{s}=(s_1,\ldots,s_n)$, where $s_k$ is defined by (\ref{eq:asad678}), is referred to as the \emph{structural signature} of the system.

Since its introduction the concept of structural signature proved to be a very useful tool in the analysis of semicoherent systems, especially
for the comparison of different system designs and the computation of the system reliability (see \cite{Sam07}).

The interest of extending the concept of signature to the general case of dependent lifetimes has been pointed out in several recent papers. Just as in the i.i.d.\ case, we can consider the $n$-tuple $\mathbf{p}=(p_1,\ldots,p_n)$, called
\emph{probability signature}, whose $k$-th coordinate is the probability $p_k=\Pr(T_S=T_{k:n})$. Thus defined, the probability signature obviously coincides with the structural signature when the component lifetimes are i.i.d.\ and continuous. Actually, it is easy to see that both concepts also coincide when the lifetimes are exchangeable and the distribution $F$ has no ties; see, e.g., \cite{NavRyc07,NavSamBalBha08} for more details. However, these two concepts are generally different. Contrary to the structural signature, the probability signature may depend on the distribution of the component lifetimes. It is then considered as a probabilistic object associated with both the structure $\phi$ and the distribution $F$; see \cite{MarMat11,MarMatWal11,Spi08,NavSpiBal10} for basic properties of this concept.

Even in the i.i.d.\ (or exchangeable) case, the computation of the signature may be a hard task when the system has a large number of components. However, the computation effort
can be greatly reduced when the system is decomposed into distinct modules (subsystems) whose structural signatures are already known.


First results along this line were presented in \cite{DaZheHu12,GerShpSpi11}. In particular, in \cite{GerShpSpi11} explicit expressions for
the structural signatures of systems consisting of two modules connected in series or in parallel were provided in terms of the structural
signatures of the modules. A general procedure to compute the structural signatures of recurrent systems (i.e., systems partitioned into
identical modules) was also described. Moreover, the key role of the concepts of tail and cumulative signatures were pointed out (see
definitions in Section 2).

In this work we extend these results in the following two directions:
\begin{enumerate}
\item[1.] Considering the general case of a system partitioned into an arbitrary number of disjoint modules connected according to an arbitrary semicoherent structure, we yield an explicit formula for the \emph{modular decomposition of the structural signature of the system}, that is, an explicit expression of the structural signature of the system only in terms of the structural signatures of the modules and the structure of the modular decomposition (i.e., the structure that defines the way the modules are interconnected). This result, which holds without any additional assumption and is obviously independent of the distribution $F$ of the component lifetimes, is presented in Section~2.

\item[2.] Considering again the general case of systems partitioned into an arbitrary number of disjoint modules, we show that a similar modular decomposition of the probability signature still holds if and only if the distribution of the component lifetimes (i.e., the function $F$) satisfies a natural decomposition condition (associated with the decomposition of the system into modules). Thus, a modular decomposition of the probability signature appears whenever two decomposition properties hold: a structural decomposition of the system into modules combined with a decomposition of the distribution of the component lifetimes. We also yield an explicit formula for this modular decomposition of the probability signature. This result is presented in Section~3. Also, we note that the proofs of our decomposition formulas are simpler than those in \cite{DaZheHu12,GerShpSpi11}.
\end{enumerate}

It is noteworthy that both the structural and probability signatures of the system can be computed by our modular decomposition formulas without knowing the structures of the modules. Only the knowledge of the signatures of the modules and the structure of the modular decomposition (i.e., the way the modules are connected) is required. Thus, the computation of the signature of a large system can be made much easier when it is decomposed into a small number of modules whose signatures are known.

In Section 4 we discuss and demonstrate our results through a few examples and provide an interpretation of the new concept of decomposition of the distribution. Some concluding remarks are then given in Section 5.

\section{Modular decomposition of the structural signature}
\label{sec:2}

We assume that the system is partitioned into modules, which in turn can be regarded as subsystems. By exploiting formula (\ref{eq:asad678}), in this section we provide an explicit formula for the structural signature of the system in terms of the structural signature of each module and the structure of the modular decomposition of the system (see Theorem~\ref{thm:main}).

Recall first that a modular decomposition of a system $(C,\phi,F)$ into $r$ disjoint modules is given by a partition $\mathcal{C}=\{C_1,\ldots,C_r\}$
of the set of components into modular subsets\footnote{Thus the subsets $C_1,\ldots,C_r$ are such that $C=\bigcup_{j=1}^rC_j$ and $C_j\cap C_k=\varnothing$ whenever $j\not=k$.} such that
\begin{itemize}
 \item for every $j\in [r]$, the components in $C_j$ are connected in a semicoherent structure described by the function $\chi_j\colon\{0,1\}^{n_j}\to\{0,1\}$ (where $n_j=|C_j|$) and thus form the module $M_j=(C_j,\chi_j,G_j)$, where $G_j$ denotes the marginal distribution, determined by $F$, of the lifetimes of the components in $C_j$;
 \item the modules are connected according to a semicoherent system described by a structure function $\psi\colon\{0,1\}^r\to\{0,1\}$;
 \item the structure $\phi$ of the system expresses through the composition
\begin{equation}\label{eq:asd45}
\phi(\bfx) ~=~ \psi(\chi_1(\bfx^{C_1}),\ldots,\chi_r(\bfx^{C_r}))\, ,\qquad\bfx\in\{0,1\}^n,
\end{equation}
where $\bfx^{C_j}=(x_i)_{i\in C_j}$, or equivalently,
\begin{equation}\label{eq:decomp12}
\phi(A) ~=~ \psi(\chi_1(A\cap C_1),\ldots,\chi_r(A\cap C_r)),\qquad A\subseteq C,
\end{equation}
\end{itemize}
(see \cite[Chap.~1]{BarPro81} for more details). For instance, if the system consists of three serially connected modules,
 then we have the structure function $\psi(z_1,z_2,z_3)=\min(z_1,z_2,z_3)=z_1\, z_2\, z_3$ and hence
$$
\phi(\bfx) ~=~ \chi_1(\bfx^{C_1})\,\chi_2(\bfx^{C_2})\,\chi_3(\bfx^{C_3})\, .
$$


Recall also that the structural signature can be equivalently expressed through the \emph{tail signature}, a concept introduced in \cite{Bol01} and named so in \cite{GerShpSpi11}. This concept is actually algebraically more convenient than that of signature and we will use it to state our formula for the modular decomposition of the signature.

The tail (structural) signature of the system is the $(n+1)$-tuple $\overline{\mathbf{S}}=(\overline{S}_0,\ldots,\overline{S}_n)$ defined by
$\overline{S}_k ~=~\sum_{i=k+1}^ns_i$ for $0\leqslant k\leqslant n-1$ and $\overline{S}_n=0$.\footnote{Clearly, $\overline{S}_n=0$ and $\overline{S}_0=1$ do not contain any information, but are defined for convenience.} The structural signature can obviously be recovered from the tail signature by using the formula $s_k=\overline{S}_{k-1}-\overline{S}_k$, for $k\in [n]$.

\begin{remark}
Coming back to Samaniego's probabilistic definition, we may interpret the number $\overline{S}_k=\sum_{i=k+1}^n\Pr(T_S=T_{i:n})=\Pr(T_S>T_{k:n})$ as the probability that the system survives beyond the $k$-th failure (provided the component lifetimes are continuous and i.i.d.)
\end{remark}

Defining the function $q_0\colon 2^{[n]}\to [0,1]$ by
\begin{equation}\label{eq:q0dev7386}
q_0(A) ~=~ \frac{1}{{n\choose |A|}}~,
\end{equation}
by (\ref{eq:asad678}) we have
\begin{equation}\label{eq:sbark}
\overline{S}_k ~=~\sum_{\textstyle{A\subseteq C\atop |A|=n-k}}q_0(A)\,
\phi(A),\qquad 0\leqslant k\leqslant n\, .
\end{equation}

Similarly, the \emph{cumulative} (structural) \emph{signature} of the system is the $(n+1)$-tuple $\mathbf{S}=(S_0,\ldots,S_n)$ defined by
$S_k=1-\overline{S}_k=\sum_{i=1}^ks_i$ for $0\leqslant k\leqslant n$.

Finally, when the system has a modular decomposition, for every $j\in [r]$, we define the function $q_0^{C_j}\colon 2^{C_j}\to [0,1]$ by
\begin{equation}\label{eq:q0dev7386j}
q_0^{C_j}(A) ~=~ \frac{1}{{n_j\choose |A|}}~,
\end{equation}
and we denote by $\overline{\mathbf{S}}^j$ and $\mathbf{S}^j$ the tail and cumulative signatures, respectively, of module $(C_j,\chi_j,G_j)$, that is
\[
\overline{S}^j_{k} ~=~ 1-S^j_{k} ~=~ \sum_{\textstyle{A\subseteq C_j\atop |A|=n_j-k}}q_0^{C_j}(A)\,\chi_j(A)\, ,\qquad 0\leqslant k\leqslant n_j=|C_j|\, .
\]

As already mentioned, our goal here is to obtain a general formula that expresses the structural signature of
the system in terms of the structural signatures of the modules (Theorem~\ref{thm:main}). This formula was obtained in \cite{GerShpSpi11} and independently in \cite{DaZheHu12}
in the special case of two modules. For two modules connected in series, the formula can be written as follows (using our notation):
\begin{equation}\label{eq:s7dfsfdf}
\overline{S}_{n-k} ~=~ \sum_{\textstyle{0\leqslant a_1\leqslant n_1,~0\leqslant a_2\leqslant n_2\atop a_1+a_2=k}}\frac{{n_1\choose
a_1}{n_2\choose a_2}}{{n\choose k}}~\overline{S}^1_{n_1-a_1}\,\overline{S}^2_{n_2-a_2}{\,}.
\end{equation}
We immediately see that the right-hand expression in this formula is independent of the structures of the modules. This means that changing these structures while keeping the same module signatures has no effect on the signature of the system.

To extend this formula to the general case, we need to recall the concept of multilinear extension of a (pseudo)-Boolean function (see \cite{Owe72}).

\begin{definition}
The \emph{multilinear extension} of a pseudo-Boolean function $\chi\colon\{0,1\}^m\to\R$ is the polynomial function $\widehat{\chi}: [0,1]^m\to\R$ defined by
\begin{equation}\label{eq:asd456}
\widehat{\chi}(z_1,\ldots,z_m) ~=~ \sum_{B\subseteq [m]}\chi(B)\,\prod_{j\in B}z_j\,\prod_{j\in [m]\setminus B}(1-z_j)\, .
\end{equation}
\end{definition}

For instance, if $\chi$ is the structure function of a series system made up of three components, we have
\[
\widehat{\chi}(z_1,z_2,z_3) ~=~ z_1z_2z_3.
\]
Similarly, if the components are connected in parallel, we then have
\[
\widehat{\chi}(z_1,z_2,z_3) ~=~ 1-(1-z_1)(1-z_2)(1-z_3).
\]

\begin{remark}\label{rem2}
The multilinear extension of a pseudo-Boolean function $\chi\colon\{0,1\}^m\to\R$ is the unique function $\widehat{\chi}:[0,1]^m\to\R$ that has the following properties:
\begin{enumerate}
\item[(i)] $\widehat{\chi}$ is a polynomial function of degree at most one in each variable,

\item[(ii)] $\widehat{\chi}$ coincides with $\chi$ on $\{0,1\}^m$.
\end{enumerate}
In particular, Eq.~(\ref{eq:asd456}) is the natural extension of the classical formula
\begin{equation}\label{eq:asd456aa}
\chi(z_1,\ldots,z_m) ~=~ \sum_{B\subseteq [m]}\chi(B)\,\prod_{j\in B}z_j\,\prod_{j\in [m]\setminus B}(1-z_j)\, .
\end{equation}
It is well known that if a semicoherent system, with structure function $\chi$, is made up of independent components, then $\widehat{\chi}$ is precisely the \emph{reliability function} of the system. Namely, for $\mathbf{r}=(r_1,\ldots,r_m)\in [0,1]^m$, $\widehat{\chi}(\mathbf{r})$ is the reliability of the system expressed as a function of the single component reliabilities $r_1,\ldots,r_m$.
\end{remark}

Define the function $c_0\colon \prod_{i=1}^r \{0,\ldots,n_i\}\to [0,1]$ by
$$
c_0(a_1,\ldots,a_r) ~=~ \frac{{n_1\choose a_1}\cdots {n_r\choose a_r}}{{n\choose a_1+\cdots +a_r}}{\,}.
$$
Combining this function with (\ref{eq:q0dev7386}) and (\ref{eq:q0dev7386j}), we obtain
\begin{equation}\label{eq:ds7f6}
q_0(A) ~=~ c_0(|A\cap C_1|,\ldots,|A\cap C_r|)\,\prod_{i=1}^r q_0^{C_j}(A\cap C_j){\,},\qquad A\subseteq C.
\end{equation}

\begin{remark}
Even though formula (\ref{eq:ds7f6}) is trivial (since it follows immediately from the definitions of $q_0$, $q_0^{C_j}$, and $c_0$), we will see that it is actually a key result for the
decomposition of the structural signature. We will also see that this formula is at the root of the decomposition of the distribution function of the component lifetimes that we will introduce in the next section to derive the modular decomposition of the probability signature (see Definition~\ref{de:7asd6}).
\end{remark}

For every $k=0,\ldots,n$, we also introduce the following set
$$
\tk ~=~ \textstyle{\{\bfa=(a_1,\ldots,a_r)\in\N^r : 0\leqslant a_j\leqslant n_j\mbox{ for $j=1,\ldots,r$ and }\sum_{j=1}^ra_j=k\}}.
$$

The algebraic tools introduced above allow us to state and prove our main theorem, which gives an explicit expression of the system tail signature $\overline{\mathbf{S}}$ in terms of the tail signatures $\overline{\mathbf{S}}^1,\ldots,\overline{\mathbf{S}}^r$ of the modules, thus generalizing formula (\ref{eq:s7dfsfdf}).

Note that the component lifetimes are irrelevant in the results of this section (since we deal with structural signatures). To stress on this fact, we write $(C,\phi)$ and $(C_j,\chi_j)$ instead of $(C,\phi,F)$ and $(C_j,\chi_j,G_j)$, respectively.

\begin{theorem}\label{thm:main}
For every semicoherent system $(C,\phi)$ with a modular decomposition into $r$ disjoint modules $(C_j,\chi_j)$, $j=1,\ldots,r$, connected
according to a semicoherent structure $\psi$, we have
\begin{equation}\label{eq:ssd4}
\overline{S}_{n-k} ~=~ \sum_{\bfa\in \mathcal{T}_k} c_0(\bfa)\,\widehat{\psi}\big(\overline{S}^1_{n_1-a_1},\ldots,
\overline{S}^r_{n_r-a_r}\big),\,\qquad 0\leqslant k\leqslant n.
\end{equation}
\end{theorem}

\begin{proof}
By combining (\ref{eq:sbark}) with (\ref{eq:decomp12}) and (\ref{eq:asd456aa}) we have for $0\leqslant k\leqslant n$
\begin{eqnarray*}
\overline{S}_{n-k} &=& \sum_{|A|=k} q_0(A)\phi(A) ~=~ \sum_{|A|=k} q_0(A)\,\psi\big(\chi_1(A\cap C_1),\ldots,\chi_r(A\cap C_r)\big)\\
&=&  \sum_{B\subseteq [r]}\psi(B)\,\sum_{|A|=k} q_0(A)\,\prod_{j\in B}\chi_j(A\cap C_j)\,\prod_{j\in [r]\setminus B}(1-\chi_j(A\cap C_j))\, .
\end{eqnarray*}
Since $\mathcal{C}=\{C_1,\ldots,C_r\}$ is a partition of $C$, the map from $2^C$ to $\prod_{j=1}^r 2^{C_j}$ given
by
$$
A\mapsto (A\cap C_1,\ldots,A\cap C_r)
$$
is a bijection that maps $\{A:|A|=k\}$ onto $\{(A_1,\ldots,A_r):(|A_1|,\ldots,|A_r|)\in \tk\}$. Therefore, we obtain
$$
\overline{S}_{n-k} ~=~ \sum_{B\subseteq [r]}\psi(B)\,\sum_{\bfa\in \tk}\sum_{\textstyle{A_1\subseteq C_1\atop |A_1|=a_1}}
\cdots\sum_{\textstyle{A_r\subseteq C_r\atop |A_r|=a_r}} q_0\bigg(\bigcup_{j=1}^rA_j\bigg)\,\prod_{j\in B}\chi_j(A_j)\,\prod_{j\in [r]\setminus
B}(1-\chi_j(A_j))\,.
$$
By (\ref{eq:ds7f6}) the right-hand side of this expression becomes
\[
\sum_{B\subseteq [r]}\psi(B)\,\sum_{\bfa\in \tk}c_0(\bfa)\,\sum_{\textstyle{A_1\subseteq C_1\atop |A_1|=a_1}}
\cdots\sum_{\textstyle{A_r\subseteq C_r\atop |A_r|=a_r}} \prod_{j=1}^rq_0^{C_j}(A_j)\,\prod_{j\in B}\chi_j(A_j)\,\prod_{j\in [r]\setminus
B}(1-\chi_j(A_j)),
\]
or equivalently,
\[
\sum_{\bfa\in \tk}c_0(\bfa)\,\sum_{B\subseteq [r]}\psi(B)\,\prod_{j\in B}\Bigg(\sum_{\textstyle{A_j\subseteq C_j\atop |A_j|=a_j}}q_0^{C_j}(A_j)
\,\chi_j(A_j)\Bigg)\,\prod_{j\in [r]\setminus B}\Bigg(\sum_{\textstyle{A_j\subseteq C_j\atop
|A_j|=a_j}}q_0^{C_j}(A_j)\,\big(1-\chi_j(A_j)\big)\Bigg)\, .
\]
Since we have $\sum_{{A_j\subseteq C_j,
|A_j|=a_j}}q_0^{C_j}(A_j)=1$ for any $a_j$ in $\{1,\ldots, n_j\}$, we immediately obtain
$$
\overline{S}_{n-k} ~=~ \sum_{\bfa\in \tk}c_0(\bfa)\,\sum_{B\subseteq [r]}\psi(B)\,\prod_{j\in B}\overline{S}^j_{n_j-a_j}\, \prod_{i\in
[r]\setminus B}(1-\overline{S}^j_{n_j-a_j}){\,},
$$
where, by (\ref{eq:asd456}), the inner sum is precisely $\widehat{\psi}(\overline{S}^1_{n_1-a_1},\ldots,
\overline{S}^r_{n_r-a_r})$.
\end{proof}

It is clear that Theorem~\ref{thm:main} is not really useful for small systems whose signatures can be computed easily. However, we now give a small example to show how Formula~(\ref{eq:ssd4}) can be applied.

\begin{example}
Consider the six-component system indicated in Figure~\ref{fig:bs}, for which we intend to compute the tail structural signature. Consider also the partition $\mathcal{C}=(\{1,2\},\{3,4,5,6\})$ associated with the modular decomposition whose organizing structure function is given by $\psi(z_1,z_2)=z_1z_2$. The tail signatures of the modules are respectively given by
$$
\overline{S}^1 ~=~ (1,1,0),\qquad \overline{S}^2 ~=~ (1,1/2,0,0,0).
$$
We have $\overline{S}_0=1$ by definition and, using Formula~(\ref{eq:ssd4}), we obtain
\begin{eqnarray*}
\overline{S}_1 &=& c_0(2,3){\,}\overline{S}_0^1{\,}\overline{S}_1^2 + c_0(1,4){\,}\overline{S}_1^1{\,}\overline{S}_0^2 ~=~ 2/3\\
\overline{S}_2 &=& c_0(2,2){\,}\overline{S}_0^1{\,}\overline{S}_2^2 + c_0(1,3){\,}\overline{S}_1^1{\,}\overline{S}_1^2 + c_0(0,4){\,}\overline{S}_2^1{\,}\overline{S}_0^2 ~=~ 4/15.
\end{eqnarray*}
We also obtain $\overline{S}_3=\overline{S}_4=\overline{S}_5=\overline{S}_6=0$.

\setlength{\unitlength}{.06\linewidth}
\begin{figure}[htbp]\centering
\begin{picture}(14,3)
\put(2.5,0){\framebox(1,1){$2$}}\put(2.5,2){\framebox(1,1){$1$}}\put(5.5,1){\framebox(1,1){$3$}}\put(7.5,1){\framebox(1,1){$4$}}%
\put(10.5,0){\framebox(1,1){$6$}}\put(10.5,2){\framebox(1,1){$5$}}%
\put(0,1.5){\line(1,0){1.5}}\put(1.5,0.5){\line(1,0){1}}\put(1.5,2.5){\line(1,0){1}}\put(3.5,0.5){\line(1,0){1}}\put(3.5,2.5){\line(1,0){1}}%
\put(4.5,1.5){\line(1,0){1}}\put(6.5,1.5){\line(1,0){1}}\put(8.5,1.5){\line(1,0){1}}\put(9.5,0.5){\line(1,0){1}}\put(9.5,2.5){\line(1,0){1}}%
\put(11.5,0.5){\line(1,0){1}}\put(11.5,2.5){\line(1,0){1}}\put(12.5,1.5){\line(1,0){1.5}}%
\put(1.5,0.5){\line(0,1){2}}\put(4.5,0.5){\line(0,1){2}}\put(9.5,0.5){\line(0,1){2}}\put(12.5,0.5){\line(0,1){2}}%
\put(0,1.5){\circle*{0.15}}\put(14,1.5){\circle*{0.15}}
\end{picture}
\caption{A six-component system} \label{fig:bs}
\end{figure}
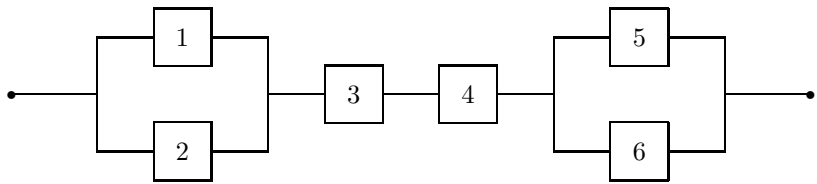
\end{example}

Interestingly, we immediately observe that the right-hand side of (\ref{eq:ssd4}) can be interpreted as an expected value with respect to the distribution defined by the function $c_0$ over $\tk$, namely the multivariate hypergeometric distribution. Note that (\ref{eq:ssd4}) can also be regarded as the law of total probability for a system whose components have i.i.d.\ lifetimes (see, in a more general setting, Proposition~\ref{prop:intc} and Remark~\ref{rem:sdf444}). This observation provides a sound, but heuristic, explanation of Theorem~\ref{thm:main}. We give however a more general and formal proof of the theorem, allowing a straightforward extension to the general case of probability signatures (see Theorem~\ref{thm:main4}).

Theorem~\ref{thm:main} immediately yields the following corollary.

\begin{corollary}\label{ref:cormain}
The structural signature of a system with a modular decomposition does not change when one modifies the modules without changing their structural signatures. In
particular, the structural signature can be computed from the structural signatures of the modules without the explicit knowledge of the structures $\chi_1,\ldots,\chi_r$ of
the modules.
\end{corollary}

\begin{example}\label{ex:11}
Suppose that the system consists of $r$ serially connected modules (hence $\psi(\mathbf{z})=\prod_{j=1}^rz_j$). Then, by (\ref{eq:ssd4}) we see
that $\overline{S}_{n-k}$ is given by the hypergeometric convolution product
$$
\overline{S}_{n-k} ~=~ \sum_{\textstyle{0\leqslant a_j\leqslant n_j\atop a_1+\cdots +a_r=k}} \frac{{n_1\choose a_1}\cdots {n_r\choose
a_r}}{{n\choose k}}\,
 \prod_{j=1}^r\overline{S}^j_{n_j-a_j}\, .
$$
\end{example}

We also obtain the following dual version of Theorem~\ref{thm:main} in which the tail signatures are replaced by the cumulative signatures.
Recall that the dual of a structure function $\chi\colon\{0,1\}^m\to\{0,1\}$ is the structure function $\chi^d\colon\{0,1\}^m\to\{0,1\}$ defined
by $\chi^d(\bfx)=1-\chi(\mathbf{1}-\bfx)$, where $\mathbf{1}=(1,\ldots,1)$.

\begin{theorem}\label{thm:maindual}
For every semicoherent system $(C,\phi)$ with a modular decomposition into $r$ disjoint modules $(C_j,\chi_j)$, $j=1,\ldots,r$, connected
according to a semicoherent structure $\psi$, we have
\begin{equation}\label{eq:ssd4dual}
S_{n-k} ~=~ \sum_{\bfa\in \tk} c_0(\bfa)\, \widehat{\psi^d}\big(S^1_{n_1-a_1},\ldots,S^1_{n_r-a_r}\big)\, .
\end{equation}
\end{theorem}

\begin{proof}
By definition of $\psi^d$ we have $\widehat{\psi^d}(z_1,\ldots,z_r)=1-\widehat{\psi}(1-z_1,\ldots,1-z_r)$, for all
$z_1,\ldots,z_r\in[0,1]$, and therefore we have
\begin{eqnarray*}
 \widehat{\psi^d}\big(S^1_{n_1-a_1},\ldots,S^r_{n_r-a_r}\big) &=& 1-\widehat{\psi}\big(1-S^1_{n_1-a_1},\ldots,1-S^r_{n_r-a_r}\big)\\
&=& 1-\widehat{\psi}\big(\overline{S}^1_{n_1-a_1},\ldots,\overline{S}^r_{n_r-a_r}\big).
\end{eqnarray*}
The right-hand side of (\ref{eq:ssd4dual}) becomes
\[
\sum_{\bfa\in \tk} c_0(\bfa)\, \left(1-\widehat{\psi}\big(\overline{S}^1_{n_1-a_1},\ldots,\overline{S}^r_{n_r-a_r}\big)\right),
\]
and the result follows from Theorem~\ref{thm:main} since the restriction of the function $c_0$ to $\tk$ is a probability distribution.
\end{proof}

\begin{example}\label{ex:144}
Suppose that the system consists of $r$ modules connected in parallel (hence $\psi^d(\mathbf{z})=\prod_{j=1}^rz_j$). Then, by
(\ref{eq:ssd4dual}), we see that $S_{n-k}$ is given by the hypergeometric convolution product
$$
S_{n-k} ~=~ \sum_{\textstyle{0\leqslant a_j\leqslant n_j\atop a_1+\cdots +a_r=k}} \frac{{n_1\choose a_1}\cdots {n_r\choose a_r}}{{n\choose k}}\,
 \prod_{j=1}^rS^j_{n_j-a_j}\, .
$$
\end{example}

\section{Modular decomposition of the probability signature}

We now analyze the problem of modular decomposition for the probability signature. Recall that the probability signature of a system $S=(C,\phi,F)$ is the $n$-tuple $\mathbf{p}=(p_1,\ldots,p_n)$ whose $k$-th coordinate is the probability $p_k=\Pr(T_S=T_{k:n})$. As already mentioned, $\mathbf{p}$ depends on $F$ and in general does not coincide with the structural signature $\mathbf{s}$.


It was shown in \cite{MarMat11} that the dependence of $\mathbf{p}$ on the c.d.f.\ $F$ is captured by the \emph{relative quality function} $q\colon 2^{[n]}\to [0,1]$ associated with $F$. This function is defined by
\begin{equation}\label{eq:dshfsdfhjn}
q(A) ~=~ \Pr\Big(\max_{i\notin A}T_i<\min_{i\in A}T_i\Big),\qquad A\subseteq C,
\end{equation}
with the convention that $q(\varnothing)=q(C)=1$ and satisfies the immediate property
\begin{equation}\label{eq:s7df5}
\sum_{|A|=k}q(A) ~=~ 1,\qquad 0\leqslant k\leqslant n.
\end{equation}
It was shown \cite{MarMat11} that, if $F$ is absolutely continuous (actually the assumption that $F$ has no ties is sufficient), then
\begin{equation}\label{eq:asghjad678}
p_k ~=~ \sum_{|A|=n-k+1}q(A)\,\phi(A)-\sum_{|A|=n-k}q(A)\,\phi(A)\, .
\end{equation}
We clearly see that (\ref{eq:asghjad678}) reduces to (\ref{eq:asad678}) whenever $q$ is a symmetric function, i.e., $q(A)=1/{n\choose |A|}=q_0(A)$, and this property holds when the component lifetimes are exchangeable.

In the general dependent case, we naturally introduce the \emph{tail probability signature} of the system as the $(n+1)$-tuple $\overline{\mathbf{P}}=(\overline{P}_0,\ldots,\overline{P}_n)$ defined by $\overline{P}_n=0$ and
$$
\overline{P}_k ~=~ \sum_{i=k+1}^np_i ~=~ \Pr(T_S>T_{k:n}) {\,},\qquad 0\leqslant k\leqslant n-1.
$$
Just as for the structural signature, the probability signature can be recovered by using the formula $p_k=\overline{P}_{k-1}-\overline{P}_k$, for $k\in [n]$.

According to (\ref{eq:asghjad678}), the analog of (\ref{eq:sbark}) for the probability signatures is
\begin{equation}\label{eq:sd5f4}
\overline{P}_k ~=~ \sum_{\textstyle{A\subseteq C\atop |A|=n-k}}q(A)\,\phi(A),\qquad 0\leqslant k\leqslant n\, .
\end{equation}

Similarly, the \emph{cumulative probability signature} of the system is the $(n+1)$-tuple $\mathbf{P}=(P_0,\ldots,P_n)$ defined by $P_k=1-\overline{P}_k=\sum_{i=1}^kp_i=\Pr(T_S\leqslant T_{k:n})$ for $0\leqslant k\leqslant n$.

As in the previous section we consider systems that are partitioned into modules. Thus, we have
a partition $\mathcal{C}=\{C_1,\ldots,C_r\}$ of the set of components and the components in each $C_j$ are connected according to a semicoherent structure $\chi_j$.

We now define the tail and cumulative probability signatures of module $(C_j,\chi_j,G_j)$, which we denote by $\overline{\mathbf{P}}^j$
and $\mathbf{P}^j$, respectively. We denote by $q^{C_j}$ the relative quality function associated with $C_j$
obtained from the marginal distribution $G_j$; that is,
$$
q^{C_j}(A) ~=~ \Pr\Big(\max_{i\in C_j\setminus A}T_i<\min_{i\in A}T_i\Big)\, ,\qquad A\subseteq C_j\, .
$$

Applying (\ref{eq:sd5f4}) in module $(C_j,\chi_j,G_j)$ we then obtain
\[
\overline{P}^j_{k} ~=~ 1-P^j_{k} ~=~ \sum_{\textstyle{A\subseteq C_j\atop |A|=n_j-k}}q^{C_j}(A)\,\chi_j(A)\, ,\qquad 0\leqslant k\leqslant n_j\, .
\]

Since the probability signatures depend on the distribution of lifetimes, we cannot expect to obtain an extension of Theorem~\ref{thm:main} to the case of probability signatures without any assumption on the function $F$. Since the signatures depend on $F$ through the relative quality functions $q$ and $q^{C_j}$ only, it is natural to impose a condition on these functions only. This can be done by extending (\ref{eq:ds7f6}) to the relative quality functions as follows.

\begin{definition}\label{de:7asd6}
Given a partition $\mathcal{C}=\{C_1,\ldots,C_r\}$ of $C$, we say that the relative quality function $q$ is \emph{$\mathcal{C}$-decomposable}
if there exists a function $c:\prod_{i=1}^r \{0,\ldots,n_i\} \to\R$ such that
\begin{equation}\label{eq:s8fd7}
q(A) ~=~ c(|A\cap C_1|,\ldots,|A\cap C_r|)\,\prod_{j=1}^rq^{C_j}(A\cap C_j)\, ,\qquad A\subseteq C\, .
\end{equation}
\end{definition}

\begin{remark}
We observe from Definition~\ref{de:7asd6} that the $\mathcal{C}$-decomposability of $q$ depends only on the partition $\mathcal{C}=\{C_1,\ldots,C_r\}$ and the function $F$ (through the functions $q$ and $q^{C_j}$) but not on the structures of the system and its modules. Moreover, the corresponding function $c$ is completely determined by the functions $q$ and $q^{C_j}$ ($j=1,\ldots,r$). Indeed, given any $r$-tuple $(a_1,\ldots,a_r)$, such that $0\leqslant a_j\leqslant |C_j|$, by (\ref{eq:s8fd7}) we have
$$
c(a_1,\ldots,a_r) ~=~ \frac{q(A_1\cup\cdots\cup A_r)}{\prod_{j=1}^rq^{C_j}(A_j)}\,
$$
whenever the subsets $A_j\subseteq C_j$ are such that $|A_j|=a_j$ and $q^{C_j}(A_j)\neq 0$ for $j=1,\ldots,r$.
To see that such subsets $A_j$ exist, just observe that Eq.~(\ref{eq:s7df5}) still holds when $q$ is replaced with $q^{C_j}$.
\end{remark}

At first glance, the condition given in Definition~\ref{de:7asd6} might seem artificial and rarely satisfied in applications. However, we will show that in a sense this condition is not only sufficient (Theorem~\ref{thm:main4}) but also necessary (Theorem~\ref{thm:f5sd7fsd5}) for the modular decomposition of the probability signature to hold. We will also provide in Section 4 a few examples in which this condition follows from natural assumptions.

We now state the extension of Theorem~\ref{thm:main} to probability signatures. Under the assumption that function $q$ is $\mathcal{C}$-decomposable, this result gives an explicit expression of the tail probability signature $\overline{\mathbf{P}}$ in terms of the tail probability signatures $\overline{\mathbf{P}}^1,\ldots,\overline{\mathbf{P}}^r$ of the modules.

\begin{theorem}\label{thm:main4}
Assume that the relative quality function $q$ associated with a distribution $F$ is $\mathcal{C}$-decomposable for some partition $\mathcal{C}=\{C_1,\ldots,C_r\}$ of $C$. Then, for every semicoherent system $(C,\phi,F)$ with a modular decomposition into $r$ disjoint modules $(C_j,\chi_j,G_j)$, $j=1,\ldots,r$, connected according to a semicoherent structure $\psi\colon\{0,1\}^r\to\{0,1\}$, we have
\begin{equation}\label{eq:ssd4proba}
\overline{P}_{n-k} ~=~ \sum_{\bfa\in \mathcal{T}_k} c(\bfa)\,\widehat{\psi}\big(\overline{P}^1_{n_1-a_1},\ldots,
\overline{P}^r_{n_r-a_r}\big),\,\qquad 0\leqslant k\leqslant n.
\end{equation}
\end{theorem}
\begin{proof}
The proof is exactly the same as that of Theorem~\ref{thm:main}, except that $q_0$ and $c_0$ must be replaced with $q$ and $c$, respectively, and Eq.~(\ref{eq:ds7f6}) with Eq.~(\ref{eq:s8fd7}).
\end{proof}

To obtain Theorem~\ref{thm:maindual}, the dual version of Theorem~\ref{thm:main}, we have used the fact that the restriction of function $c_0$ to each set $\mathcal{T}_k$ is a probability
distribution (namely the multivariate hypergeometric distribution). The following result shows that the restriction of function $c$ to each set $\mathcal{T}_k$ is also a probability
distribution whenever the relative quality function
$q$ is $\mathcal{C}$-decomposable for some partition $\mathcal{C}$.

For every $k\in\{0,\ldots,n\}$ and every $\bfa\in\tk$ we introduce the following event:
\begin{eqnarray}
E_{k,\bfa} &=& (\mbox{among the first $n-k$ failed components, there are exactly}\label{eq:f68sa7dg}\\
 && \mbox{ $n_j-a_j$ components in $C_j$ for all $j\in [r]$}).\nonumber
\end{eqnarray}
We observe that $E_{k,\bfa}$ is also the following event: (among the best $k$ components, there are exactly $a_j$ components in $C_j$ for all $j\in [r]$).

\begin{proposition}\label{prop:intc}
Assume that the relative quality function $q$ is $\mathcal{C}$-decomposable for some partition $\mathcal{C}=\{C_1,\ldots,C_r\}$ of $C$.
Then, for each $k\in \{0,\ldots,n\}$ the restriction of function $c$ to $\tk$ is a probability distribution.
More precisely, for every $\bfa\in\tk$, $c(\bfa)$ is exactly the probability $\Pr(E_{k,\bfa})$.
\end{proposition}

\begin{proof}
Since $q(A)$ is the probability that the best $|A|$ components are precisely those in $A$, we must have
\begin{equation}\label{eq:r65w7rwe}
\Pr(E_{k,\bfa}) ~=~ \sum_{A\subseteq C:{\,}|A\cap C_1|=a_1,\ldots,|A\cap C_r|=a_r}q(A).
\end{equation}

Since $q$ is $\mathcal{C}$-decomposable, by (\ref{eq:s8fd7}) we have
\begin{eqnarray*}
\Pr(E_{k,\bfa}) &=& \sum_{A_1\subseteq C_1:{\,}|A_1|=a_1}\,\cdots\,\sum_{A_r\subseteq C_r:{\,}|A_r|=a_r}c(\bfa)\,\prod_{j=1}^rq^{C_j}(A_j)\\
&=& c(\bfa)\,\prod_{j=1}^r\bigg(\sum_{A_j\subseteq C_j:{\,}|A_j|=a_j}q^{C_j}(A_j)\bigg),
\end{eqnarray*}
where the product is $1$ (apply Eq.~(\ref{eq:s7df5}) to every function $q^{C_j}$).
\end{proof}

\begin{remark}\label{rem:sdf444}
We note that a formula similar to (\ref{eq:ssd4proba}) can be derived from the law of total probability. In fact, since the events $E_{k,\bfa}$ ($\bfa\in\tk$), as defined right before Proposition~\ref{prop:intc}, form a partition of the sample space, we have
$$
\overline{P}_{n-k} ~=~ \Pr(E_k) ~=~ \sum_{\bfa\in \mathcal{T}_k} \Pr(E_k\mid E_{k,\bfa})\,\Pr(E_{k,\bfa}),
$$
where
\begin{eqnarray*}
E_k &=& (T>T_{n-k:n})\\ &=& (\mbox{The system is still surviving after the first $(n-k)$ component failures}).
\end{eqnarray*}
On the one hand, we saw in (\ref{eq:r65w7rwe}) that $\Pr(E_{k,\bfa})$ is a sum of $q$ values and reduces to $c(\bfa)$ under the $\mathcal{C}$-decomposability of $q$ (Proposition~\ref{prop:intc}). On the other hand, one can show that, under the conditional independence of the events $(T_{C_j}>T_{n-k:n})$, $j=1,\ldots,r$, given $E_{k,\bfa}$, where $T_{C_j}$ is the lifetime of module $(C_j,\chi_j,G_j)$, we obtain
\begin{equation}\label{eq:6fsd7sdds}
\Pr(E_k\mid E_{k,\bfa}) ~=~ \widehat{\psi}\big(\Pr(T_{C_1}>T_{n-k:n}\mid E_{k,\bfa}),\ldots,\Pr(T_{C_r}>T_{n-k:n}\mid E_{k,\bfa})\big).
\end{equation}
Finding general conditions under which the probability $\Pr(T_{C_j}>T_{n-k:n}\mid E_{k,\bfa})$ reduces to $\overline{P}_{j,n_j-a_j}$ remains an interesting open question.
\end{remark}

We now provide the dual form of Theorem~\ref{thm:main4}, that is the extension of Theorem~\ref{thm:maindual} to the probability signatures. The proof is similar to that of Theorem~\ref{thm:maindual} and thus is omitted.

\begin{theorem}\label{thm:maindual2}
Assume that the relative quality function $q$ associated with a distribution $F$ is $\mathcal{C}$-decomposable for some partition $\mathcal{C}=\{C_1,\ldots,C_r\}$ of $C$. Then, for every semicoherent system $(C,\phi,F)$ with a modular decomposition into $r$ disjoint modules $(C_j,\chi_j,G_j)$, $j=1,\ldots,r$, connected according to a semicoherent structure $\psi\colon\{0,1\}^r\to\{0,1\}$, we have
\begin{equation}\label{eq:ssd4ddualproba}
P_{n-k} ~=~ \sum_{\bfa\in \tk} c(\bfa)\, \widehat{\psi^d}\big(P^1_{n_1-a_1},\ldots,P^r_{n_r-a_r}\big)\, .
\end{equation}
\end{theorem}


We end this section by showing that in a sense the $\mathcal{C}$-decomposability of $q$ is necessary for the modular decomposition of the probability signature to hold.

\begin{theorem}\label{thm:f5sd7fsd5}
Consider a partition $\mathcal{C}=\{C_1,\ldots,C_r\}$ of $C$ and a distribution $F$ of the component lifetimes. Assume that there exists a function $\gamma\colon\prod_{j=1}^r\{0,\ldots,n_j\}\to\R$ such that, for every semicoherent system $(C,\phi,F)$ with a modular decomposition into $r$ disjoint modules $(C_j,\chi_j,G_j)$, $j=1,\ldots,r$, connected according to a semicoherent structure $\psi\colon\{0,1\}^r\to\{0,1\}$, we have
\begin{equation}\label{eq:d4}
\overline{P}_{n-k} ~=~ \sum_{\bfa\in \mathcal{T}_k} \gamma(\bfa)\, \widehat{\psi}\big(\overline{P}^1_{n_1-a_1},\ldots,\overline{P}^r_{n_r-a_r}\big)\,.
\end{equation}
Then the relative quality function $q$ associated with $F$ is $\mathcal{C}$-decomposable.
\end{theorem}

\begin{proof}
Let us consider a subset $B$ of the set of components $C$ and try to decompose $q(B)$. The key observation is that the relative quality function is determined by
 the tail signature of appropriate systems. Those systems are only semicoherent.

If $B$ is empty, then $q(B)=1$ (by definition) and, since $B\cap C_1,\ldots,B\cap C_r$ are also empty, we have $q^{C_j}(B\cap C_j)=1$.
Therefore we can set $c(0,\ldots,0)=1$.

If $B$ is nonempty, then so is the set $m_B=\{j\in [r]:B\cap C_j\not=\varnothing\}$. For every $j\in m_B$, let us form the module $(C_j,\chi_j,G_j)$, where $\chi_j\colon 2^{C_j}\to\{0,1\}$ is defined by
\[
\chi_j(D) ~=~
\left\{\begin{array}{ll}
1, & \mbox{if}\,D\supseteq B\cap C_j{\,},\\
0, & \mbox{otherwise}.
\end{array}\right.
\]
In other words, the pseudo-Boolean function $\chi_j:\{0,1\}^{n_j}\to\R$ is given by
\[\chi_j(y_1,\ldots,y_{n_j})=\prod_{k\in B\cap C_j}y_k.\]
It is semicoherent since $B\cap C_j$ is nonempty.

Now, let us connect these modules in series to obtain a structure $\phi$
that is also semicoherent (this means that we consider $\psi(z_1,\ldots,z_r)=\prod_{j\in m_B}z_j$). Actually, the function $\phi$ is nothing other than the set function $\phi_B\colon 2^{C}\to\{0,1\}$ defined by
\[
\phi_B(D) ~=~
\left\{\begin{array}{ll}
1, & \mbox{if}\,D\supseteq B,\\
0, & \mbox{otherwise}.
\end{array}\right.
\]
Now, we compute the tail probability signatures of such functions using the formula
$$
\overline{P}_{n-k} ~=~ \sum_{\textstyle{A\subseteq C\atop |A|=k}}q(A)\,\phi(A)\, .
$$
It follows that for $0\leqslant k<|B|$, we have $\overline{P}_{n-k}=0$. Moreover, setting $k=|B|$, we have
\[
\overline{P}_{n-k} ~=~ \overline{P}_{n-|B|} ~=~ \sum_{\textstyle{A\subseteq C\atop |A|=|B|}}q(A)\,\phi_B(A) ~=~
\sum_{\textstyle{A\supseteq B\atop |A|=|B|}}q(A)\,\phi_B(A) ~=~ q(B).
\]
We obtain the same results for the module signatures: for every $j\in m_B$, we have
$\overline{P}^j_{n_j-k}=0$ for $0\leqslant k<|B\cap C_j|$ and
\[
\overline{P}^j_{n_j-|B\cap C_j|} ~=~ q^{C_j}(B\cap C_j).
\]
Now, applying formula (\ref{eq:d4}) to the system $(C,\phi,F)$ with $k=|B|$ leads to
\begin{equation}\label{eqq}
q(B) ~=~ \sum_{\bfa\in \mathcal{T}_{|B|}} \gamma(\bfa)\, \widehat{\psi}\big(\overline{P}^1_{n_1-a_1},\ldots,\overline{P}^r_{n_r-a_r}\big)
 ~=~ \sum_{\bfa\in \mathcal{T}_{|B|}} \gamma(\bfa)\, \prod_{j\in m_B}\overline{P}^j_{n_j-a_j}.
\end{equation}
We notice that the latter sum contains only one nonzero term. Indeed, by definition we have
\[\mathcal{T}_{|B|}=\textstyle{\{\bfa=(a_1,\ldots,a_r)\in\N^r : 0\leqslant a_j\leqslant n_j\mbox{ for $j=1,\ldots,r$ and }\sum_{j=1}^ra_j=|B|\}}\, \]
but in Eq.~(\ref{eqq}) the product corresponding to an $\bfa\in\mathcal{T}_{|B|}$ is zero unless $a_j\geqslant |B\cap C_j|=b_j$ for all $j\in m_B$,
and we obviously have $a_j\geqslant b_j$ for $j\not\in m_B$. Taking the condition $\sum_{j=1}^ra_j=|B|$ into account, the only term that does not vanish corresponds to
$a_j=b_j$ for all $j\in m_B$, and $a_j=0=b_j$ for every $j\not\in m_B$, which yields
\[
q(B) ~=~ \gamma(b_1,\ldots,b_r)\prod_{j\in m_B}q^{C_j}(B\cap C_j) ~=~ \gamma(b_1,\ldots,b_r)\prod_{j=1}^rq^{C_j}(B\cap C_j)
\]
and thus completes the proof.
\end{proof}

The fact that $\mathcal{C}$-decomposability of $q$ is both necessary and sufficient to ensure the modular decomposition of the probability signature motivates the following definition, which extends the concept of modular decomposition to the general non-i.i.d.\ case.

\begin{definition}
We say that a semicoherent system $(C,\phi,F)$ is \emph{decomposable} if, for some partition $\mathcal{C}=\{C_1,\ldots,C_r\}$ of $C$ with $1<r<n$,
\begin{enumerate}
\item[(i)] the structure $\phi\colon\{0,1\}^n\to\{0,1\}$ has a modular decomposition into $r$ disjoint semicoherent modules $(C_j,\chi_j,G_j)$, $j=1,\ldots,r$, and

\item[(ii)] the function $q$ is $\mathcal{C}$-decomposable.
\end{enumerate}
\end{definition}

We note that Theorem~\ref{thm:main4} gives a very compact and quite
explicit formula for the tail probability signature of a decomposable system.

\section{Applications and interpretations}

The results obtained in the previous sections require some comments in terms of applications and interpretations. In this section we first show the impact of our main theorems through a few natural examples. Then, observing that the $\mathcal{C}$-decomposability of function $q$ is a key property in the main results, we show its relevance by providing some natural situations where it holds (without requiring that $q$ is symmetric) and give a possible interpretation of it as a form of independence.

%
%
%

\subsection{Applications of the main theorems}

In addition to Examples~\ref{ex:11} and \ref{ex:144}, we give here a few more examples of how the main theorems can be applied in usual situations. For the sake of simplicity, we focus here on the structural signature.

\begin{example}\label{ex:E}
Let $(C_1,\chi_1)$,...,$(C_r,\chi_r)$ be $r$ modules consisting of serially connected components. The tail structural signature of $(C_j,\chi_j)$ is given by
$\overline{S}^j_{0}=1$ and $\overline{S}^j_{l}=0$ for $1\leqslant l\leqslant n_j$ and $1\leqslant j\leqslant r$.
We observe that the tuples $\overline{\mathbf S}^{j}$ are Boolean. Therefore, in Formula (\ref{eq:ssd4}) we can use $\psi$ in place of $\widehat{\psi}$ and we obtain
\[
\overline{S}_{n-k} ~=~ \sum_{\bfa\in{\mathcal T}_k}\frac{{n_1\choose a_1}\cdots {n_r\choose a_r}}{{n\choose k}}\,\psi(\uno_{\{a_1=n_1\}},\ldots,\uno_{\{a_r=n_r\}}){\,},\qquad
0\leqslant k\leqslant n,
\]
where $\uno_{\{p\}}$ denotes the truth value of proposition $p$.

Recalling (\ref{eq:asd456aa}), this formula can also be written as
\[
\overline{S}_{n-k} ~=~ \sum_{\bfa\in{\mathcal T}_k}\frac{{n_1\choose a_1}\cdots {n_r\choose a_r}}{{n\choose k}}\sum_{B\subseteq [r]}\psi(B)
 \prod_{j\in B}\uno_{\{a_j=n_j\}}\prod_{j\in [r]\setminus B}\uno_{\{a_j<n_j\}}.
\]
As a special case, assume that the modules are connected in parallel. Then $\psi$ is the maximum function and we obtain
\[
\overline{S}_{n-k}~=~\sum_{\bfa\in{\mathcal V}_k}\frac{{n_1\choose a_1}\cdots {n_r\choose a_r}}{{n\choose k}}{\,},
\]
where ${\mathcal V}_k=\{\bfa\in\tk{\,}\mid{\,} a_j=n_j~\mbox{for some}\,j\in [r]\}$.
\end{example}

We can extend the previous example by considering modules of $k_j$-out-of-$n_j$ type.\footnote{Recall that a $k$-out-of-$n$ system is an $n$-component system that is in a failed state if and only if at least $k$ components are in a failed state (such a system is also referred to as a $k$-out-of-$n{\colon}F$ system). In particular a series system is a $1$-out-of-$n$ system.}

\begin{example}\label{ex:F}
For each $j\in[r]$, let $(C_j,\chi_j)$ be a $k_j$-out-of-$n_j$ system, with $1\leqslant k_j\leqslant n_j$, and let $\psi$ be the organizing structure function of these modules. The tail structural signature of $(C_j,\chi_j)$ is
then given by $\overline{S}^j_{l}=1$ if $0\leqslant l<k_j$ and $\overline{S}^j_{l}=0$ otherwise. Using Theorem \ref{thm:main}, we obtain (as in Example~\ref{ex:E})
\[
\overline{S}_{n-k} ~=~ \sum_{\bfa\in{\mathcal T}_k}\frac{{n_1\choose a_1}\cdots {n_r\choose a_r}}{{n\choose k}}\,\psi(\uno_{\{a_1>n_1-k_1\}},\ldots,\uno_{\{a_r>n_r-k_r\}}){\,},\qquad
0\leqslant k\leqslant n.
\]
When the modules are connected in parallel, $\psi$ is the maximum function and we obtain
\[
\overline{S}_{n-k} ~=~ \sum_{\bfa\in{\mathcal V}_k}\frac{{n_1\choose a_1}\cdots {n_r\choose a_r}}{{n\choose k}}{\,},
\]
where this time ${\mathcal V}_k=\{\bfa\in\tk{\,}\mid {\,}a_j>n_j-k_j~\mbox{for some}\,j\in [r]\}$.
\end{example}

We observe that formulas in Examples \ref{ex:E} and \ref{ex:F} also hold for the tail probability signatures whenever the relative quality function is decomposable for the considered partition, up to replacement of the multivariate hypergeometric coefficients ${n_1\choose a_1}\cdots {n_r\choose a_r}/{n\choose k}$ by the coefficients $c(a_1,\ldots,a_r)$ associated with the decomposition of $q$.

Different papers, starting from \cite{KocMukSam99}, have shown that stochastic comparisons between lifetimes of two systems can be established in
terms of their signatures. These results can be combined with ours in different forms. In the following example we consider the analysis
of redundancy.

\begin{example}\label{example:G}
We consider a coherent structure $\chi\colon\{0,1\}^n\to\{0,1\}$ and two disjoint sets of components $C'$ and $C''$, each containing $n$ components. The components are assumed to have i.i.d.\ lifetimes.

Having at our disposal the two sets $C'$ and $C''$, we can build redundant structures. A classical problem amounts to determining the optimal way to
arrange redundancies. In particular one can compare \emph{redundancy at system level} with \emph{redundancy at component level}. More formally, starting from $\chi$, let us consider the two structures $\phi_1,\phi_2\colon\{0,1\}^{2n}\to\{0,1\}$ defined as follows for $\bfx,\bfy\in\{0,1\}^n$:
\[
\phi_1(\bfx, \bfy)=\max(\chi(\bfx),\chi(\bfy))\quad\mbox{and}\quad \phi_2(\bfx,\bfy)=\chi(\max(x_1,y_1),\ldots,\max(x_n,y_n)).
\]
It is a well-known fact (see e.g.~\cite{BarPro81}) that the structure $\phi_2$
is more reliable than $\phi_1$. Our arguments can allow us to obtain some more
detailed results along this direction. In fact, once the signatures of $\phi_1$ and $\phi_2$
have been computed, results presented in \cite{KocMukSam99} can be applied
to obtain different stochastic comparisons between the lifetimes associated with $\phi_1$ and $\phi_2$.

Let ${\mathbf S}=(S_0,\ldots,S_n)$ denote the cumulative signature associated with $\chi$. We now apply Theorem \ref{thm:maindual}
 to compute the cumulative signatures associated with $\phi_1$ and $\phi_2$. For the first one, the structure inside each module is $\chi$
and the modules are connected in parallel. Thus the organizing function $\psi$ in Theorem  \ref{thm:maindual} is the maximum function (in two variables),
whose dual $\psi^d$ is the minimum function. We thus have $\widehat{\psi^d}(s,t)=s{\,}t$ for $s,t\in[0,1]$, and obtain the formula
\[
S_{2n-k}^{(1)}~=~\sum_{a=\max(0,n-k)}^{\min(n,2n-k)}\frac{{n \choose a}{n\choose 2n-k-a}}{{2n\choose k}}\, S_aS_{2n-k-a}{\,},\qquad 0\leqslant k\leqslant 2n.
\]
For the second one, the structure inside each module is parallel, so the associated cumulative structural signature is ${\mathbf S}^j=(0,0,1)$ for $j=1,\ldots,n$.
Moreover the modules are connected according to $\chi$. Therefore we have
\[
S_{2n-k}^{(2)}~=~\sum_{\bfa\in\tk}\frac{{2\choose a_1}\cdots{2\choose a_n}}{{2n\choose k}}\, \chi^d(\uno_{\{a_1=0\}},\ldots,\uno_{\{a_n=0\}}){\,},
\qquad 0\leqslant k\leqslant 2n.
\]
\end{example}

\subsection{On the $\mathcal{C}$-decomposability of the relative quality function}

We examine some natural cases where the relative quality function $q$ is $\mathcal{C}$-decomposable for some partition $\mathcal{C}=\{C_1,\ldots,C_r\}$.

The conditions on the probability distribution of lifetimes for $q$ to be $\mathcal{C}$-decompo{\-}sable express through a finite number of consistent equations on the values of the functions $q$ and $q^{C_j}$. It is known that the values of these functions depend on the distribution of lifetimes only through the values $p_{\sigma}=\Pr(T_{\sigma(1)}<\cdots <T_{\sigma(n)})$ $(\sigma\in S_n)$, where $S_n$ is the set of permutations on $[n]$. In fact we have (see \cite{MarMat11})
$$
q(A) ~= \sum_{\sigma\in S_n{\,}:{\,}\{\sigma(n-|A|+1),\ldots,\sigma(n)\}=A}p_{\sigma}
$$
and a similar formula can be derived for $q^{C_j}$ (see Example~\ref{ex:20-19}).

Therefore, (\ref{eq:s8fd7}) is a set of conditions on the values $p_{\sigma}$ $(\sigma\in S_n)$. Every family $(p_{\sigma}:\sigma\in S_n)$ satisfying (\ref{eq:s8fd7}) provides infinitely many distributions of lifetimes for which $q$ is $\mathcal{C}$-decomposable: $\mathcal{C}$-decomposability holds regardless of the probability laws $\mathcal{L}(T_1,\ldots,T_n\mid T_{\sigma(1)}<\cdots <T_{\sigma(n)})$. Clearly, setting $p_{\sigma}=1/n!$, the functions $q$ and $q^{C_j}$ ($j=1,\ldots,r$) reduce to $q_0$ and $q_0^{C_j}$, respectively, and hence $q$ is $\mathcal{C}$-decomposable. We thus have an infinite family of distributions of lifetimes for which $q$ is $\mathcal{C}$-decomposable. This family subsumes the i.i.d.\ and exchangeable cases.

Let us now consider situations where the relative quality function $q$ may be different from $q_0$ and is $\mathcal{C}$-decomposable for some partition $\mathcal{C}=\{C_1,\ldots,C_r\}$.

%

\begin{definition}
We say that the function $q\colon 2^C\to\R$ is \emph{$\mathcal{C}$-symmetric} for a partition $\mathcal{C}=\{C_1,\ldots,C_r\}$ if $q(A)=q(B)$ for every $A,B\subseteq C$ such that $|A\cap C_j|=|B\cap C_j|$ for every $j\in [r]$.
\end{definition}

In other terms, $q$ is $\mathcal{C}$-symmetric if and only if $q(A)$ depends on $A$ only through the numbers $|A\cap C_1|,\ldots,|A\cap C_r|$.

\begin{proposition}\label{prop:yx87c}
If the function $q$ is $\mathcal{C}$-symmetric for some partition $\mathcal{C}=\{C_1,\ldots,C_r\}$ and the functions $q^{C_j}$ are symmetric for $j=1,\ldots,r$, then $q$ is $\mathcal{C}$-decomposable.
\end{proposition}

\begin{proof}
Under the assumptions of the proposition, we have $q^{C_j}(B)=1/{n_j\choose |B|}$ for all $B\subseteq C_j$ and all $j\in [r]$. Hence (\ref{eq:s8fd7}) holds since the expression
$$
\frac{q(A)}{\prod_{j=1}^rq^{C_j}(A\cap C_j)} ~=~ q(A) \,\prod_{j=1}^r {n_j\choose |A\cap C_j|}\, ,\qquad A\subseteq C\, ,
$$
depends only on the $r$-tuple $(|A\cap C_1|,\ldots,|A\cap C_r|)$.
\end{proof}

The special case where the functions $q$ and $q^{C_j}$ $(j=1,\ldots,r)$ are symmetric is presented in the next corollary.

\begin{corollary}\label{cor:87sdf6}
If $q$ and $q^{C_j}$ $(j=1,\ldots,r)$ are symmetric, then $q$ is $\mathcal{C}$-decomposable for every partition $\mathcal{C}=\{C_1,\ldots,C_r\}$ and we have
\begin{equation}\label{eq:ddf5}
c(\bfa) ~=~ \frac{{n_1\choose a_1}\cdots {n_r\choose a_r}}{{n\choose a_1+\cdots +a_r}}\, .
\end{equation}
\end{corollary}

The following example shows that the assumptions of Proposition~\ref{prop:yx87c} do not imply those of Corollary~\ref{cor:87sdf6}. Similarly, as already mentioned, the assumptions of Corollary~\ref{cor:87sdf6} do not imply that the component lifetimes are i.i.d.\ or exchangeable.

\begin{example}
Let us consider a system made up of three components, with lifetimes $T_1,T_2,T_3$, such that for every permutation $\sigma$ on $\{1,2,3\}$, we have
\[
\Pr(T_{\sigma(1)}<T_{\sigma(2)}<T_{\sigma(3)}) ~=~
\left\{\begin{array}{ll}1/4 &\mbox{if}~\sigma(3)=3,\\ 1/8 & \mbox{otherwise}.
       \end{array}
\right.\]
Let us show that the function $q$ is $\mathcal{C}$-symmetric for the partition $\mathcal{C}=(\{1,2\},\{3\})$. We have
$$
q(\{1\}) ~=~ \Pr(T_2<T_3<T_1) + \Pr(T_3<T_2<T_1) ~=~ \frac{1}{4}{\,}.
$$
Similarly, $q(\{2\})=1/4$, $q(\{3\})=1/2$, and $q(\{1,3\})=q(\{2,3\})=3/8$. We also have
$$
q^{\{1,2\}}(\{1\}) ~=~ \Pr(T_2<T_1) ~=~ \frac{1}{2} ~=~ \Pr(T_1<T_2) ~=~ q^{\{1,2\}}(\{2\}).
$$
It follows that the assumptions of Proposition~\ref{prop:yx87c} are satisfied. However, $q$ is not symmetric.
\end{example}

Let us now show that the assumptions of Proposition~\ref{prop:yx87c} are not necessary for $q$ to be $\mathcal{C}$-decomposable. In the next example the function $q$ is $\mathcal{C}$-decomposable for a given partition $\mathcal{C}$ but it is not $\mathcal{C}$-symmetric.

\begin{example}\label{ex:20-19}
Let us consider a system made up of three components, with lifetimes $T_1,T_2,T_3$, such that for every permutation $\sigma$ on $\{1,2,3\}$, we have
\[
\Pr(T_{\sigma(1)}<T_{\sigma(2)}<T_{\sigma(3)}) ~=~
\left\{\begin{array}{ll}2/9 &\mbox{if}~\sigma^{-1}(1)<\sigma^{-1}(2),\\ 1/9 & \mbox{otherwise}.
       \end{array}
\right.\]
In other words, we have $\Pr(T_{\sigma(1)}<T_{\sigma(2)}<T_{\sigma(3)})=2/9$ if and only if the event $(T_{\sigma(1)}<T_{\sigma(2)}<T_{\sigma(3)})$ is included in the event $(T_1<T_2)$.

We consider the partition $\mathcal{C}=(\{1,2\},\{3\})$. By using the definitions of functions $q$ and $q^{\{1,2\}}$, we obtain for instance
$$
q(\{1\}) ~=~ \Pr(T_2<T_3<T_1)+\Pr(T_3<T_2<T_1)~=~ \frac{2}{9}
$$
and
$$
q^{\{1,2\}}(\{1\})~=~\Pr(T_2<T_1<T_3)+\Pr(T_2<T_3<T_1)+\Pr(T_3<T_2<T_1)~=~ \frac{1}{3}{\,}.
$$

Similarly, we obtain the following values:
$$
q(\{1\}) ~=~ q(\{1,3\}) ~=~ \frac{2}{9}{\,},\quad q(\{3\}) ~=~ q(\{1,2\})~=~ q^{\{1,2\}}(\{1\})~=~ \frac{1}{3}{\,},
$$
$$
q(\{2\}) ~=~ q(\{2,3\}) ~=~ \frac{4}{9}{\,},\quad q^{\{1,2\}}(\{2\})~=~ \frac{2}{3}{\,}.
$$
It is then easy to see that the function $q$ is $\mathcal{C}$-decomposable for the partition $\mathcal{C}=(\{1,2\},\{3\})$ by showing that (\ref{eq:s8fd7}) is satisfied. For instance, we must have
$$
q(\{1\}) ~=~ c(1,0){\,}q^{\{1,2\}}(\{1\}){\,}q^{\{3\}}(\varnothing)
$$
and
$$
q(\{2\}) ~=~ c(1,0){\,}q^{\{1,2\}}(\{2\}){\,}q^{\{3\}}(\varnothing),
$$
which are satisfied when $c(1,0)=2/3$. However, $q$ is not $\mathcal{C}$-symmetric since we have $q(\{1\})=2/9$ and $q(\{2\})=4/9$.
\end{example}

Several results obtained in reliability theory require the independence of the component lifetimes. Since our results are based on the $\mathcal{C}$-decomposability of $q$, it is natural to interpret this property as a weak form of independence. In this respect, we have the following result, which makes use of the events $E_{k,\bfa}$ defined in (\ref{eq:f68sa7dg}).

\begin{proposition}\label{cor:intc}
Let $\mathcal{C}=\{C_1,\ldots,C_r\}$ be a partition of $C$. Then $q$ is $\mathcal{C}$-decomposable if and only if, for every $A\subseteq C$, we have $\Pr(E_{k,\bfa})=0$ (with $k=|A|$ and $a_j=|A\cap C_j|$) or
$$
\Pr\Big(\max_{i\notin A}T_i<\min_{i\in A}T_i{\,}\Big|{\,}E_{k,\bfa}\Big)~=~ \prod_{j=1}^rq^{C_j}(A\cap C_j).
$$
\end{proposition}

\begin{proof}
On the one hand, by Proposition~\ref{prop:intc}, $q$ is $\mathcal{C}$-decomposable if and only if
$$
q(A) ~=~ \Pr(E_{k,\bfa})\,\prod_{j=1}^rq^{C_j}(A\cap C_j)\, ,\qquad A\subseteq C\, .
$$
On the other hand, we always have
$$
q(A) ~=~  \Pr(E_{k,\bfa})\,\Pr\Big(\max_{i\notin A}T_i<\min_{i\in A}T_i{\,}\Big|{\,}E_{k,\bfa}\Big)\, ,\qquad A\subseteq C\, ,
$$
which completes the proof.
\end{proof}

Proposition~\ref{cor:intc} says that $q$ is $\mathcal{C}$-decomposable if and only if, for every $A\subseteq C$, if $\Pr(E_{k,\bfa})\neq 0$ (with $k=|A|$ and $a_j=|A\cap C_j|$), then the probability that the best $k$ components are precisely those in $A$ knowing that among them there are exactly $a_j$ components in $C_j$, $j\in [r]$, factorizes as the product over $j\in [r]$ of the probabilities that the best $a_j$ components in $C_j$ are precisely those in $A\cap C_j$. Thus, the $\mathcal{C}$-decomposability of $q$ turns out to be a form of independence.

\section{Conclusion}

The main purpose of this paper is to analyze the computation of the signature of a system decomposed into disjoint modules in terms of the signatures of these modules. This problem was considered recently in \cite{DaZheHu12} and independently in \cite{GerShpSpi11} in the special case of a $2$-module system with i.i.d.\ or exchangeable lifetimes, that is, for the structural signature. We provide here the most general result for computing the structural signature of an arbitrary system decomposed into disjoint modules in terms of the corresponding module signatures (Theorem~\ref{thm:main}), thus fully answering the question raised in \cite{DaZheHu12} and \cite{GerShpSpi11}. We observe that our derivations are substantially based on formula (\ref{eq:asad678}). Such a formula, which was first pointed out in \cite{Bol01}, ties the concept of structural signature with the family of path sets, which constitutes a classical tool in the reliability analysis of coherent systems. In fact, formula (\ref{eq:asad678}) shows that only the numbers of the path sets of the different sizes are relevant in the computation of the tail structural signature.

This general answer is further extended to the general non-i.i.d.\ case in Theorem~\ref{thm:main4} for the computation of the probability signature of the system in terms of the probability signatures of the modules. In this general case one cannot expect to obtain a general decomposition result that would hold without any assumption on the distribution of the lifetimes. This is why we introduce here a concept of factorization for the quality function associated with a system decomposed into disjoint modules (Definition~\ref{de:7asd6}). This concept might be seen as a special notion of \emph{partial exchangeability} defined in terms of the relative quality function $q$. Even though this property may seem artificial or unnatural at first glance, Theorems~\ref{thm:main4} and \ref{thm:f5sd7fsd5} together show that in a sense it is actually necessary and sufficient for decomposing the probability signature. Finally, we present natural examples of non-exchangeable distributions of lifetimes where this factorization condition holds.

We also note that our approach allows us to treat the general non-i.i.d.\ case of arbitrary systems with proofs that are much simpler than those used in \cite{DaZheHu12} to deal with the case of $2$-module systems with i.i.d.\ lifetimes.%

\section*{Acknowledgments}

This paper was started during a one-week visit of Jean-Luc Marichal and Pierre Mathonet at the Department of
 Mathematics of the University La Sapienza, Rome, Italy. This opportunity is greatly acknowledged.
Fabio Spizzichino acknowledges financial support of Research Project Modelli e Algoritmi Stocastici, University La Sapienza, 2009. Jean-Luc Marichal acknowledges partial support by the internal research project F1R-MTH-PUL-12RDO2 of the University of Luxembourg.

\end{document}